\documentclass[12pt]{article}
\usepackage{amsmath}
\usepackage{amsthm}
\usepackage{amssymb}
\usepackage{latexsym}
\usepackage{color}
\usepackage{graphicx}
\usepackage[all,cmtip]{xy}
\usepackage{stmaryrd}
\usepackage{tikz-cd}
\usetikzlibrary{decorations.pathmorphing}
\usepackage{mathrsfs}

\usepackage[colorlinks=true, linkcolor=blue, citecolor=red]{hyperref}

\newcommand{\C}{\mathbb{C}}

\newcommand{\PP}{\mathbb P}

\renewcommand{\mod}{\mathrm{mod}}

\DeclareMathOperator{\Coh}{\mathrm{Coh}}
\DeclareMathOperator{\proj}{\mathrm{proj}}

\DeclareMathOperator{\inj}{\mathrm{inj}}
\DeclareMathOperator{\length}{\mathrm{length}}
\DeclareMathOperator{\GP}{\mathrm{GP}}

\DeclareMathOperator{\Hom}{Hom}
\DeclareMathOperator{\Spec}{\mathrm{Spec}}

\DeclareMathOperator{\Pic}{Pic}

\DeclareMathOperator{\tor}{tor}

\DeclareMathOperator{\CM}{CM}
\DeclareMathOperator{\SCM}{SCM}
\DeclareMathOperator{\Per}{Per}
\DeclareMathOperator{\Sing}{Sing}
\DeclareMathOperator{\add}{add}
\DeclareMathOperator{\sg}{sg}

\newcommand{\mc}[1]{\mathcal{#1}}
\newcommand{\mk}[1]{\mathfrak{#1}}
\newcommand{\mb}[1]{\mathbb{#1}}
\newcommand{\msr}[1]{\mathscr{#1}}

\newcommand{\wt}{\widetilde}

\newcommand{\Span}[1]{\left<#1\right>}

\newcommand{\End}{\operatorname{\mathrm{End}}}
\newcommand{\Ker}{\operatorname{\mathrm{Ker}}}
\renewcommand{\Im}{\operatorname{\mathrm{Im}}}

\newcommand{\Cok}{\operatorname{\mathrm{Coker}}}

\newcommand{\id}{\mathrm{id}}

\newcommand{\rank}{\mathrm{rank}}

\newcommand{\Ext}{\mathop{\mathrm{Ext}}\nolimits}
\newcommand{\ind}{\operatorname{\mathrm{ind}}}

\newcommand{\RHom}{\mathop{\mathbb R\mathrm{Hom}}\nolimits}
\newcommand{\mcRHom}{\mathop{\mathbb R\mathcal{H}om}\nolimits}
\newcommand{\mcHom}{\mathop{\mathcal{H}om}\nolimits}
\newcommand{\mcExt}{\mathop{\mathcal{E}xt}\nolimits}

\mathchardef\mhyphen="2D

\theoremstyle{plain}
\newtheorem{theorem}{Theorem}[section]
\newtheorem{lemma}[theorem]{Lemma}
\newtheorem{proposition}[theorem]{Proposition}
\newtheorem{corollary}[theorem]{Corollary}
\newtheorem{claim}[theorem]{Claim}

\theoremstyle{definition}
\newtheorem{definition}[theorem]{Definition}
\newtheorem{example}[theorem]{Example}
\newtheorem{remark}[theorem]{Remark}

\title{Special Cohen--Macaulay sheaves on partial resolutions of rational surfaces singularities}
\author{Hokuto Uehara}
\date{}
\begin{document}

\maketitle

\begin{abstract}
We introduce the categories $\CM(X)$ and $\SCM(X)$ of reflexive sheaves on a minimal partial resolution $f\colon X \to \Spec R$ of a rational surface singularity $\Spec R$. The main result of this paper establishes that $\SCM(X)$ possesses a natural Frobenius structure, serving as a geometric counterpart to the algebraic Frobenius structure on special Cohen--Macaulay $R$-modules, introduced by Iyama--Wemyss and Iyama--Kalck--Wemyss--Yang. Utilizing this geometric framework, we establish an exact equivalence between $\SCM(X)$ and the category of special Cohen--Macaulay $R$-modules equipped with a specific exact structure, which induces a triangle equivalence between their stable categories. Consequently, this provides a direct, geometric proof of the Iyama--Kalck--Wemyss--Yang equivalence and yields a Buchweitz-type equivalence $\underline{\SCM}(X) \simeq D_{\sg}(X)$.
\end{abstract}

\renewcommand{\thefootnote}{}
\renewcommand{\thefootnote}{\arabic{footnote}}

\section{Introduction}
The interplay between the geometry of resolutions of singularities and the category of Cohen--Macaulay modules is a fundamental theme in algebraic geometry and representation theory. The significance of special Cohen--Macaulay (SCM) modules in this context is highlighted by two pivotal, yet distinct, developments. 

On the geometric side, the Wunram correspondence \cite{MR926422} establishes a bijection between the irreducible exceptional curves $\{E_i\}_{i \in I}$ of the minimal resolution of a quotient surface singularity over $\mb{C}$ and the nontrivial indecomposable SCM modules over the coordinate ring. This serves as a module-theoretic analogue of the McKay correspondence at the level of objects. 

On the homological side, Iyama, Kalck, Wemyss, and Yang \cite{MR3320570} uncovered a deep categorical structure governed by SCM modules. Given a minimal partial resolution $f\colon X \to \Spec R$ of a rational surface singularity, they introduced a natural exact structure on the category of SCM modules and established a triangle equivalence between the resulting stable category and the singularity category of $X$.

A key technical ingredient in our approach is the study of globally generated reflexive sheaves on $X$, subject to specific vanishing conditions on extensions. Specifically, we define the following subcategories of $\Coh (X)$:
\begin{align*}
\CM(X) &:= \{ \mc{F} \mid \text{$\mc{F}$ is a globally generated reflexive sheaf, } \Ext_X^1(\mc{F},\omega_X)=0 \},\\
\SCM(X) &:= \{ \mc{F} \mid \text{$\mc{F}$ is a globally generated reflexive sheaf, } \Ext_X^1(\mc{F},\mc{O}_X)=0 \},\\
\mk{V}(X) &:= \{ \mc{F} \mid \text{$\mc{F}$ is a globally generated locally free sheaf, } \Ext_X^1(\mc{F},\mc{O}_X)=0 \}.
\end{align*}

By studying reflexive pullbacks $\wt{M}^X:=f^*M/\tor f^*M$ on $X$ for $M\in \SCM (R)$, our approach bypasses the use of singularity categories of $X$ entirely and relies directly on the geometric structure of partial resolutions. The central theme of this paper is to uncover the exact and triangulated structures of these categories. 

As a preparatory step, we formulate a Wunram-type correspondence for minimal partial resolutions (Proposition \ref{prop:Wunram_partial}). While this object-wise bijection can be essentially deduced from \cite[Theorem 4.3]{MR3886189}, it serves as a fundamental building block for our framework. Specifically, there is a one-to-one correspondence between the set of exceptional divisors $\{E_i\}_{i \in I_f}$ of the minimal partial resolution $f$ and the set of isomorphism classes of nontrivial indecomposable SCM $R$-modules $M$ such that $\wt{M}^X$ is locally free on $X$. We denote by $N_f$ the direct sum of these SCM $R$-modules together with $R$.

As illustrated below, these ingredients form a three-way correspondence:
\begin{equation}\label{eqn:wunram-type0}
\begin{tikzcd}[row sep=15pt, column sep=15pt]
  & \{E_i\}_{i\in I} \arrow[dl, leftrightarrow, "\text{Wunram-type}" sloped] \arrow[dr, leftrightarrow] & \\
  \ind\SCM(R) \setminus\{R\} & & \ind\SCM(X)\setminus\{\mc{O}_X\} \arrow[ll, leftrightarrow, "\text{Lemma \ref{lem:special_partial}}"'{pos=0.25}] \\
  & \{E_i\}_{i\in I_f} \arrow[dl, leftrightarrow, "\text{Proposition \ref{prop:Wunram_partial}}" sloped, pos=0.4] \arrow[dr, leftrightarrow, 
  sloped, pos=0.4] \arrow[uu, hook] & \\
  \ind(\add N_{f})\setminus\{R\} \arrow[uu, hook] & & \ind(\add \wt{N_f})\setminus \{\mc{O}_X\} \arrow[ll, leftrightarrow] \arrow[uu, hook]
\end{tikzcd}
\end{equation}

Building on this setup, our main original results unfold in the following two aspects:

\vspace{0.5em}
\noindent\textbf{1. A Frobenius structure on $\SCM(X)$.} 
The highlight and main result of this paper is the discovery that the category $\SCM(X)$ naturally inherits a Frobenius structure. This provides a purely geometric realization of the Frobenius structures previously studied in algebraic settings by Iyama and Wemyss \cite{MR3006691}. 

\begin{theorem}[Theorem \ref{thm:Frob_SCM}]
The exact category $\SCM(X)$ is a Frobenius category. Its projective-injective objects are precisely given by $\add \wt{N_{f\circ g^c}}$, where $g^c$ is the canonical model of the minimal resolution $g\colon Y\to X$.
\end{theorem}

When $f = \id$ (so that $X = \Spec R$), this recovers the Frobenius structure on $\SCM(R)$ established in \cite{MR3006691}.

\vspace{0.5em}
\noindent
\vspace{0.5em}
\noindent
\textbf{2. Exact equivalences and geometric proofs of triangle equivalences.} 
Under the assumption that $X$ has rational double points—which is essential  in \cite{MR3320570} to equip $\SCM_{N_f}(R)$ with a Frobenius structure having $\proj \SCM_{N_f}(R) = \add N_f$  (and forces $g^c = \id$ so that $N_{f\circ g^c} = N_f$)—we categorify the three-way correspondence \eqref{eqn:wunram-type0} as follows:
\begin{theorem}[Theorem \ref{thm:DX}]
The assignment $M \mapsto \wt{M}^X$ yields an exact equivalence between the exact category $\SCM_{N_f}(R)$ and $\SCM(X)$, which restricts to an equivalence between $\add N_f$ and $\add \wt{N_f}$. Consequently, this induces a triangle equivalence between their respective stable categories:
\[
\underline{\SCM}_{N_f}(R) \xrightarrow{\sim} \underline{\SCM}(X).
\]
\end{theorem}

This geometric perspective bridges algebraic representations and sheaf theory on minimal partial resolutions, immediately yielding a Buchweitz-type equivalence without requiring the intermediate algebraic machinery used in previous works:
\begin{theorem}[Theorem \ref{thm:buchweitz}]
There is a triangle equivalence between $\underline{\SCM}(X)$ and $D_{\sg}(X)$.
\end{theorem}

Finally, by composing Theorem \ref{thm:DX} with the pullback functor to the formal completions at the singular points of $X$, we obtain a direct, geometric proof of the triangle equivalence originally established in characteristic $0$ by Iyama, Kalck, Wemyss, and Yang  in \cite{MR3320570}, confirming their result in arbitrary characteristic:
\begin{theorem}[Corollary \ref{cor:IKWY}]
There is a triangle equivalence 
\[
\underline{\SCM}_{N_f}(R) \xrightarrow{\sim} \underline{\CM}(T),
\]
where $T$ is the product of the formal completions at the singular points of $X$.
\end{theorem}

\vspace{1em}
The paper is organized as follows. In Section 2, we establish the foundations of reflexive pullbacks. In Section 3, we introduce the categories $\CM(X)$ and $\SCM(X)$, and generalize the Wunram-type correspondence to minimal partial resolutions. In Section 4, we prove our main result by endowing $\SCM(X)$ with a Frobenius structure. Finally, in Section 5, we establish the exact and triangle equivalences, yielding the Buchweitz-type equivalence and the geometric proof of the IKWY theorem.

\paragraph{Acknowledgements}
Some of the ideas in this article were conceived during the author's visit to University of Graz in 2025. He is grateful to Martin Kalck for his warm hospitality and for many stimulating discussions that helped shape this work. He also thanks Wahei Hara for stimulating discussion. 
He is supported by the Grants-in-Aid for Scientific Research (No.~23K03074). 

\paragraph{AI discloser}
During the preparation of this work, the author used Google Gemini in order to translate parts of the draft into English, proofread the text, and improve overall readability. After using this tool/service, the author reviewed and edited the content as needed and takes full responsibility for the content of the publication.

\paragraph{Convention}
Throughout the paper, $k$ denotes an algebraically closed field of arbitrary characteristic. 
We assume that all rings $R$ and schemes $X$ are Noetherian. We also assume that all $R$-modules are finitely generated and all sheaves on $X$ are coherent. 

We implicitly identify coherent sheaves on $\Spec R$ with finitely generated $R$-modules under the equivalence of abelian categories $\Coh(\Spec R) \cong \mod R$. For a sheaf $\mc{F}$ on a scheme $X$, $\tor \mc{F}$ denotes the torsion part of $\mc{F}$. 

By a \textit{partial resolution} of a normal scheme $X$, we mean a proper birational morphism $g\colon Y \to X$ from a normal scheme $Y$.
A partial resolution is \textit{minimal} if $\omega_Y$ is $g$-nef.

By a Cohen--Macaulay $R$-module, we always mean a maximal Cohen--Macaulay module.
We denote by $\CM (R)$ the category of Cohen--Macaulay $R$-modules. 

For an exact category $\msr{E}$, $\inj \msr{E}$ (respectively, $\proj \msr{E}$) denotes the full subcategory consisting of injective objects (respectively, projective objects). 

For an additive category $\msr{A}$,
let $\ind \msr{A}$ denote the set of isomorphism classes of indecomposable objects in $\msr{A}$.

For objects $\mc{F}, \mc{G}$ of a triangulated category $\msr{T}$ and $n \in \mb{Z}$, we set 
$$
\Ext^n_{\msr{T}}(\mc{F}, \mc{G}) := \Hom_{\msr{T}}(\mc{F}, \mc{G}[n]).
$$

\section{Preliminaries}\label{sec:preliminaries}
\subsection{Rational surface singularities in arbitrary characteristic}

Let $X$ be an excellent normal scheme. In characteristic zero, $X$ is said to have rational singularities if there exists a resolution of singularities $g\colon Y\to X$ such that $\mb{R}^i g_* \mc{O}_Y = 0$ for all $i > 0$. By the Grauert--Riemenschneider vanishing theorem, this automatically implies that $X$ is Cohen--Macaulay and $\mb{R}^i g_* \omega_Y = 0$ for all $i>0$.

In arbitrary characteristic, particularly in dimensions $\ge 3$, the failure of the Grauert--Riemenschneider vanishing theorem necessitates a more careful definition. Following Kov\'acs \cite{MR3057950}, a resolution $g\colon Y\to X$ is called a \textit{rational resolution} if $\mb{R}^i g_* \mc{O}_Y = 0$ and $\mb{R}^i g_* \omega_Y = 0$ for all $i > 0$, and $X$ is defined to have rational singularities if such a rational resolution exists. 

However, throughout this paper, we focus exclusively on dimension two (e.g., normal surfaces over a field, or spectra of normal two-dimensional complete local rings). In this case, these characteristic-dependent subtleties disappear. 

\begin{remark}\label{rem:rational_surface_char_p}
For two-dimensional excellent normal schemes, resolutions of singularities exist. Furthermore, by \cite[Theorem 10.4]{MR3057950}, the single condition $\mb{R}^1 g_* \mc{O}_Y = 0$ alone guarantees both that $X$ is Cohen--Macaulay and that $\mb{R}^1 g_* \omega_Y = 0$. Thus, in dimension two, the classical definition suffices in any characteristic. 

Moreover, as is well known, if $X$ has rational surface singularities, this vanishing condition naturally extends to any partial resolution $h\colon Z \to X$. Namely, by applying the Leray spectral sequence and Zariski's Main Theorem to a resolution dominating $Z$, one immediately obtains $\mb{R}h_*\mc{O}_Z \cong \mc{O}_X$. 
It also follows that $Z$ has rational singularities.
\end{remark}


\subsection{Reflexive sheaves on surfaces}
For a sheaf $\mc{F}$ on a normal scheme $X$, its dual is defined as $\mc{F}^\vee:=\mcHom_X(\mc{F},\mc{O}_X)$, and $\mc{F}$ is called reflexive if the natural morphism $\mc{F}\to \mc{F}^{\vee\vee}$ is an isomorphism.

In what follows, we assume that $X$ is a $2$-dimensional normal scheme. Let us summarize several basic facts which we will freely use in the subsequent calculations:
\begin{itemize}
    \item  A sheaf $\mc{F}$ on $X$ is reflexive if and only if it is Cohen--Macaulay (i.e., $\mc{F}_x$ is a Cohen--Macaulay $\mc{O}_{X,x}$-module for every $x\in X$).
    \item  Any reflexive sheaf $\mc{F}$ is locally free on the smooth locus of $X$. Consequently, for any integer $i>0$ and any sheaf $\mc{F}'$, the support of $\mcExt_X^i(\mc{F},\mc{F}')$ is of dimension $\le 0$.
    \item  For any integer $i>0$, $\mcExt_X^i(\mc{F},\omega_X)=0$ if $\mc{F}$ is reflexive. Similarly, $\mcExt_X^i(\mc{F}, \mc{F}') = 0$  for any sheaf $\mc{F}'$ if $\mc{F}$ is locally free.
    \item Let $f\colon X\to \Spec H^0(X,\mc{O}_X)$ be the natural morphism. In our setting, since the fibers of $f$ have dimension at most $1$, the higher direct image $\mb{R}^2f_*$ vanishes for any coherent sheaf. Thus, for any sheaves $\mc{F}, \mc{F}'$ on $X$, the local-to-global spectral sequence yields a short exact sequence
    \[
        0 \to \mb{R}^1f_*\mcHom_X(\mc{F}, \mc{F}') \to \Ext^1_X(\mc{F}, \mc{F}') \to f_*\mcExt^1_X(\mc{F}, \mc{F}') \to 0.
    \]
    Consequently, $\Ext_X^1(\mc{F}, \mc{F}') = 0$ holds if and only if
    \[
        \mb{R}^1f_*\mcHom_X(\mc{F}, \mc{F}') = 0 \quad \text{and} \quad f_*\mcExt^1_X(\mc{F}, \mc{F}') = 0.
    \]
    In our arguments, we will primarily apply this equivalence when $\mc{F}$ is a reflexive sheaf, evaluating the local term $f_*\mcExt^1_X(\mc{F}, \mc{F}')$ via the support and vanishing properties described above.
\end{itemize}

\begin{lemma}\label{lem:surjection}
Let $f\colon X\to W$ be a birational projective morphism between $2$-dimensional normal schemes, and $\mc{F}$ be a reflexive sheaf on $X$ satisfying $\mb{R}^1f_*(\mc{F}^\vee)=0$.  
Suppose that $\mc{G}$ is $f$-globally generated.
Then 
\[
\mb{R}^1f_*\mcHom_X(\mc{F},\mc{G})=0.
\]
\end{lemma}

\begin{proof} We may assume $W$ is affine. Since there is a surjection 
$\mc{O}_X^{\oplus r}\to \mc{G}$ for some integer $r>0$, the natural morphism
\[\phi\colon \mcHom_X(\mc{F},\mc{O}_X^{\oplus r})\to \mcHom_X(\mc{F},\mc{G})\]
is surjective outside the singular points of $X$. Since the fibers of $f$ have dimension at most $1$, we have $\mb{R}^2f_* = 0$. Hence, there are exact sequences
\begin{align*}
&\mb{R}^1f_*\mcHom_X(\mc{F},\mc{O}_X^{\oplus r})\to \mb{R}^1f_*\Im \phi \to 0,\\
&\mb{R}^1f_*\Im \phi\to \mb{R}^1f_*\mcHom_X(\mc{F},\mc{G})\to \mb{R}^1f_*\Cok \phi\to 0.
\end{align*}
Since the support of $\Cok \phi$ is at most $0$-dimensional, we have $\mb{R}^1f_*\Cok \phi=0$.
Hence, the vanishing of $\mb{R}^1f_*(\mc{F}^\vee)$ implies the result.
\end{proof}

\subsection{Generalities for reflexive pull-backs}
Throughout this subsection, 
let $g\colon Y\to X$ 
be a birational projective morphism between 
excellent $2$-dimensional normal schemes with rational singularities. In particular, we have $\mb{R}g_*\mc{O}_Y\cong \mc{O}_X$
and $g$ has at most $1$-dimensional fibers. 

For $\mc{F}\in \Coh (X)$, we set $\widetilde{\mc{F}}^Y:=g^*\mc{F}/\tor g^*\mc{F}$. If it is clear from the context, we suppress the superscript $Y$
 and write $\widetilde{\mathcal{F}}$ instead of $\widetilde{\mathcal{F}}^Y$.

\begin{lemma}\label{lem:XY}
Let $f\colon X\to W$ be a partial resolution. 
For any $\mc{G}\in \Coh (W)$, we have $\widetilde{\widetilde{\mc{G}}^X}^Y\cong \widetilde{\mc{G}}^Y$.
\end{lemma}

\begin{proof}
From the definitions, we obtain the following commutative diagram of short exact sequences:
\begin{equation*}
\xymatrix{
                     & C_1                                                         &                       0                    &    C_3                                                         & \\
0 \ar[r]           &  \tor g^*\widetilde{\mc{G}}^X  \ar[r]\ar[u] &g^*\widetilde{\mc{G}}^X\ar[r]\ar[u]&\widetilde{\widetilde{\mc{G}}^X}^Y \ar[r]\ar[u] &0 \\
0 \ar[r]          &  \tor g^*f^*\mc{G} \ar[r]\ar[u]_{\alpha_1}      & g^*f^*\mc{G}\ar[r]\ar[u]_{\alpha_2} & \widetilde{\mc{G}}^Y \ar[r]\ar[u]_{\alpha_3} &0 \\
                     & K_1       \ar[u]                                      & K_2 \ar[u]                      & K_3 \ar[u]                                           &\\
}
\end{equation*}
Here $C_i:=\Cok \alpha_i$ and $K_i:=\Ker \alpha_i$. 
By the snake lemma, we see that $C_3=0$, and obtain an exact sequence
\[
0\to K_1\to K_2\to K_3\to C_1\to0.
\]
Moreover, $K_2$ and $C_1$ are torsion sheaves. Thus, the exact sequence  implies that $K_3$ is a torsion sheaf.
On the other hand, $K_3$ is a subsheaf of the torsion-free sheaf $\wt{\mc{G}}^Y$.
Hence $K_3=0$, and consequently $\alpha_3$ is an isomorphism.
\end{proof}

For a globally generated sheaf $\mc{G}$ on $Y$,
there is a surjection $\mc{O}_Y^{\oplus r}\to \mc{G}$ for some integer $r>0$. Since $\mb{R}^1g_*\mc{O}_Y=0$ and 
the fibers of $g$ are at most $1$-dimensional, $\mb{R}^1g_*\mc{G}=0$. We will use this fact implicitly hereafter.

(Versions of) Lemmas \ref{lem:pullback} and \ref{lem:pushforward} can be found in \cite[2.1, 2.2]{MR0809966} and in the proof of \cite[Proposition 2.7]{MR3886189}.

\begin{lemma}\label{lem:pullback}
Let $\mathcal{F}$ be a sheaf on $X$. 
Then we have the following.
\begin{enumerate}
\item
If $\mc{F}$ is globally generated, so is
$\widetilde{\mathcal{F}}$.
\item
If $\mc{F}$ is reflexive, then
$g_*\widetilde{\mathcal{F}}\cong \mathcal{F}$.
\item
If $\mc{F}$ is a globally generated reflexive sheaf, then
$\widetilde{\mathcal{F}}$ is reflexive.
\end{enumerate} 
\end{lemma}

\begin{proof}
(i) is obvious. (ii) There is a natural injective map
$\mc{F}\hookrightarrow g_*\wt{\mc{F}}$.
Since its cokernel is supported in dimension at most $0$, and $\mc{F}$ is reflexive, it is an isomorphism.

(iii) By applying $g_*$ to the short exact sequence
$$
0\to \widetilde{\mc{F}}\stackrel{\alpha}\to\widetilde{\mc{F}}^{\vee\vee}\to \Cok \alpha\to 0,
$$ 
we have
$$
0\to g_*\widetilde{\mc{F}}\cong\mc{F}\stackrel{g_*\alpha}\to g_*(\widetilde{\mc{F}}^{\vee\vee})\to g_*(\Cok \alpha)\to \mb{R}^1g_*\widetilde{\mc{F}}.
$$
Since $\Cok \alpha$ and hence $g_*(\Cok \alpha)$ are supported in dimension at most $0$, and moreover $\mc{F}$ is reflexive, we note that $g_*\alpha$ is an isomorphism. 
Moreover $\mb{R}^1g_*\widetilde{\mc{F}}=0$ by (i), which implies  $g_*(\Cok \alpha)=0$. Therefore we have $\Cok \alpha=0$. 
This gives the conclusion.
\end{proof}

\begin{lemma}\label{lem:pushforward}
Let $\mc{G}$ be a reflexive sheaf on $Y$.
\begin{enumerate}
    \item 
If $\Ext^1_Y(\mc{G},\omega_Y)=0$, then $g_*\mc{G}$ is reflexive.
\item 
If $\mc{G}$ is $g$-globally generated, then $\wt{g_*\mc{G}}\cong \mc{G}$, and we have 
$$
\Ext^1_Y(\mc{G},\omega_Y)\cong \Ext^1_X(g_*\mc{G},\omega_X).
$$
\end{enumerate}
\end{lemma}

\begin{proof}
(i)
We may assume that $X$ is the spectrum of a $2$-dimensional complete local normal domain $(R,\mk{m})$.
We have an exact sequence
\[
0=H^0_{\mk{m}}(g_*\mc{G})\to g_*\mc{G}\to i_*i^*(g_*\mc{G})\to H^1_{\mk{m}}(g_*\mc{G}),
\]
where $i\colon \Spec R\setminus \{\mk{m}\} \hookrightarrow \Spec R$.
Local duality implies
\[\mb{R}\Gamma_\mk{m}(\mb{R}g_*\mc{G})\cong \Hom_R(\RHom_R(\mb{R}g_*\mc{G},\omega_R[2]), E(R/\mk{m})),\]
where $E(R/\mk{m})$ is the injective hull of $R/\mk{m}$.
Since
\[\RHom_R(\mb{R}g_*\mc{G},\omega_R[2])\cong  
\RHom_Y(\mc{G},g^!\omega_R[2])
\cong  
\RHom_Y(\mc{G},\omega_Y[2]),
\]
if $\Ext_Y^1(\mc{G},\omega_Y)$ vanishes, then $H^1_{\mk{m}}(g_*\mc{G})=0$, which implies that  $g_*\mc{G}$ is reflexive.

(ii) There exists a surjective and generically injective map $g^*g_*\mc{G}\to\mc{G}$, which yields  $\wt{g_*\mc{G}}\cong \mc{G}$.
Moreover, recall that $\mb{R}g_*\mc{G}\simeq g_*\mc{G}$. By Grothendieck duality (i.e., the adjunction $\mb{R}g_* \dashv g^!$) and the fact that $g^!\omega_X \simeq \omega_Y$, we obtain the quasi-isomorphism:
\[\mb{R}\Hom_Y(\mc{G},\omega_Y) \simeq \mb{R}\Hom_X(\mb{R}g_*\mc{G},\omega_X) \simeq \mb{R}\Hom_X(g_*\mc{G},\omega_X).\]
Taking the first cohomology yields the second isomorphism $\Ext^1_Y(\mc{G},\omega_Y)\cong \Ext^1_X(g_*\mc{G},\omega_X)$.

\end{proof}

\begin{lemma}\label{lem:omega}
Let $f\colon X\to \Spec R$ be a minimal partial resolution of a rational surface singularity $\Spec R$. Then $\omega_X\cong \wt{\omega_R}$.  
\end{lemma}
\begin{proof}
Take a minimal resolution $g\colon Y\to X$. Then, since $f\circ g$ is minimal, $\omega_Y$ is an $f\circ g$-nef invertible sheaf. Thus, it is $f\circ g$-globally generated by \cite[Theorem 12.1]{MR276239}. Then, since $(f\circ g)_*\omega_Y\cong \omega_R$, we have $\omega_Y \cong \wt{\omega_R}^Y$ by Lemma \ref{lem:pushforward}(ii). Moreover, $\wt{\omega_R}^X$ is reflexive by Lemma \ref{lem:pullback}(iii).   
Therefore, we obtain
\[
\omega_X\cong g_*\omega_Y\cong g_*\wt{\omega_R}^Y \cong g_*\wt{\wt{\omega_R}^X}^Y\cong \wt{\omega_R}^X
\]
by Lemmas \ref{lem:XY} and \ref{lem:pullback}(ii).
\end{proof}


\section{Wunram-type results for minimal partial resolutions}\label{sec:Wunram-type}
Let $(R,\mk{m})$ be a $2$-dimensional complete local normal domain with a rational singularity over $k$.
Let
\[
Y \xrightarrow{g} X \xrightarrow{f} \Spec R
\]
be the minimal resolution of $\Spec R$, which factors through a partial resolution $f$. In this situation, $f$ is also minimal.

The set of $(-2)$-curves among $\{E_i\}_{i\in I}$ forms a disjoint union of $\mathrm{ADE}$ configurations. When $X$ has at worst rational double points, the partial resolution $f$ factors through $X^c$, which is minimal among partial resolutions of $\Spec R$ with only rational double points:
\[ X\to X^c \xrightarrow{f^c} \Spec R. \]
$X^c$ is called the \emph{canonical model} of $f$.

\subsection{Generalizations of special Cohen--Macaulay modules to sheaves on minimal partial resolutions}
We apply the results in the previous section.

Denote by
\[
\{E_i\}_{i\in I}
\]
the set of irreducible exceptional divisors of the minimal resolution
$f\circ g$.

Among these exceptional divisors, some are contracted by $g$ and some are not.
We define
\[
I_g:=\{i\in I \mid E_i \text{ is exceptional for } g\},
\]
and
\[
I_f:=I\setminus I_g.
\]

For each $i\in I_f$, the image $g(E_i)$ is an irreducible exceptional divisor of $f$.
Identifying $E_i$ with its image on $X$, we may regard
\begin{equation}\label{eqn:If}
\{E_i\}_{i\in I_f}
\end{equation}
as the set of irreducible exceptional divisors of $f$.

First, we recall the definition of special Cohen--Macaulay modules.

\begin{definition}
A Cohen--Macaulay $R$-module $M$ is said to be \textit{special} if $\Ext^1_R(M,R)=0$. We denote the category of special Cohen--Macaulay $R$-modules by $\SCM(R)$. 
\end{definition}

We generalize the notions of Cohen--Macaulay and special Cohen--Macaulay modules to globally generated reflexive sheaves on a minimal partial resolution $X$ as follows:
\begin{align*}
\CM(X) &:= \left\{\mc{F} \;\middle|\; 
  \begin{aligned}
    &\text{$\mc{F}$ is a globally generated reflexive sheaf on $X$,}\\ 
    &\Ext_X^1(\mc{F},\omega_X)=0
  \end{aligned}
\right\},\\
\SCM(X) &:= \left\{\mc{F} \;\middle|\; 
  \begin{aligned}
    &\text{$\mc{F}$ is a globally generated reflexive sheaf on $X$,}\\ 
    &\Ext_X^1(\mc{F},\mc{O}_X)=0
  \end{aligned}
\right\},\\
\mk{V}(X) &:= \left\{\mc{F} \;\middle|\; 
  \begin{aligned}
    &\text{$\mc{F}$ is a globally generated locally free sheaf on $X$,}\\ 
    &\Ext_X^1(\mc{F},\mc{O}_X)=0
  \end{aligned}
\right\}.
\end{align*}

\begin{remark}\label{rem:CMSCM}
\begin{enumerate}
 \item
        Since locally free sheaves are reflexive, we have $\mk{V}(X)\subset \SCM (X)$.
        Moreover, since $\omega_X$ is globally generated by Lemma \ref{lem:omega}, 
        we have $\SCM (X)\subset \CM (X)$ by Lemma \ref{lem:surjection}:
        \[\mk{V}(X)\subset \SCM (X)\subset \CM (X).\]
\item 
Reflexive sheaves are locally free on the smooth surface $Y$. Hence,  $\mk{V}(Y)=\SCM (Y)$.  
\item
Note that $\Coh (X)$ is a Krull--Schmidt abelian category because $X$ is projective over $\Spec R$ for a complete local ring $R$.
Furthermore, $\mk{V}(X)$, $\SCM (X)$, and $\CM (X)$ are subcategories of $\Coh (X)$ that are closed under extensions and direct summands; hence, they are Krull--Schmidt exact categories.
\item 
        Recall that a module over a $2$-dimensional normal local domain is reflexive if and only if it is Cohen--Macaulay. Moreover, since any reflexive $R$-module $M$ satisfies $\Ext_R^1(M,\omega_R)=0$, we have $\CM(R)=\CM(\Spec R)$ and $\SCM(R)=\SCM(\Spec R)$. 
        \item 
        While any reflexive $R$-module $M$ satisfies $\Ext_R^1(M,\omega_R)=0$, the analogous statement $\Ext_X^1(\mc{F},\omega_X)=0$ does not necessarily hold for a globally generated reflexive sheaf $\mc{F}$ on $X$. Indeed, a sufficiently ample invertible sheaf on $X$ provides a counterexample.
    \end{enumerate}
\end{remark}

\begin{lemma}\label{lem:vanishing_extension}
Let $h\colon Z\to W$ be a partial resolution of a rational surface singularity $W$, and define a subcategory of $\Coh(Z)$ by
\[
\msr{C}:=\{\mc{F}\in \Coh (Z) \mid \mb{R}h_* \mc{F}=0\}.
\]
Then the following hold:
\begin{enumerate}
    \item $\msr{C}=\Span{\mc{O}_{E_i}(-1) \mid \text{ $E_i$ is an irreducible exceptional divisor of $h$ }}_{\text{ext}}$.
    \item Assume furthermore that $h$ is a minimal resolution. Then for any $\mc{F} \in \msr{C}$, we have $\Ext^1_Z(\mc{F},\mc{O}_Z)=0$.
\end{enumerate}
\end{lemma}

\begin{proof}
Recall that every irreducible exceptional divisor $E$ of $h$ is a projective line $\PP^1$ (cf. \cite[Lemma 3.4.1]{MR2057015}).

(i) Although this is well-known, we provide a sketch of the proof for the reader's convenience. First, note that any $\mc{F}\in \msr{C}$ is a pure $1$-dimensional sheaf supported on the exceptional divisors. Then there is a surjective map from $\mc{F}$ to a locally free sheaf on some $E_i$. Since $\mc{F}\in \msr{C}$, this yields a surjective map from $\mc{F}$ to $\mc{O}_{E_i}(-1)$. The kernel of this map also belongs to $\msr{C}$, so the result follows by induction on $\sum_i \length _{\eta_i}\mc{F}$, where $\eta_i$ is the generic point of $E_i$.

(ii) By (i), it suffices to show $\Ext^1_Z(\mc{O}_{E}(-1),\mc{O}_Z)=0$ for any irreducible exceptional divisor $E$. Suppose that $E^2=-n$. Since $\omega_Z$ is $h$-nef, we have $n\ge 2$. For a closed immersion $i\colon E\hookrightarrow Z$, we have
\[i^!\mc{O}_Z=\omega_{E/Z}[-1]\cong \mc{O}_E(-n)[-1]\]
by \cite[Corollary 7.3]{MR222093}. Then by Grothendieck duality,
\begin{align*}
\mcRHom_Z(i_*\mc{O}_{E}(-1),\mc{O}_Z)&\cong i_*\mcRHom_E(\mc{O}_E(-1),i^!\mc{O}_Z)\\
&\cong i_*\mc{O}_E(-n+1)[-1].
\end{align*}
Taking the first cohomology, we obtain $\mcExt^1_Z(\mc{O}_{E}(-1),\mc{O}_Z)\cong i_*\mc{O}_E(-n+1)$. Since $\mcHom_Z(\mc{O}_E(-1), \mc{O}_Z) = 0$, the local-to-global spectral sequence implies
\[\Ext^1_Z(\mc{O}_{E}(-1),\mc{O}_Z) \cong H^0(Z, i_*\mc{O}_E(-n+1)) \cong H^0(E,\mc{O}_E(-n+1)). \]
Since $n \ge 2$, this group vanishes, which concludes the proof.
\end{proof}

 The correspondence between $\CM (X)$ and $\CM (R)$ is shown in \cite[Proposition 2.7]{MR3886189}. We show that the correspondence restricts to one between $\SCM (X)$ and $\SCM (R)$.
 
\begin{lemma}\label{lem:special_partial}
There is an additive equivalence
\[
\begin{matrix}
\CM (X) & \longleftrightarrow & \CM(R) \\[1ex]
\mc{F} & \longmapsto & f_*\mc{F} \\
\wt{M} & \longmapsfrom & M,
\end{matrix}
\]
which restricts to an additive equivalence between $\SCM(X)$ and $\SCM (R)$. 
\end{lemma}

\begin{proof}
The equivalence between $\CM (X)$ and $\CM (R)$ is shown in
\cite[Proposition 2.7]{MR3886189}. 
It also follows from Lemmas \ref{lem:pullback} and \ref{lem:pushforward}.

\begin{claim} 
Let $h\colon Z\to W$ be a partial resolution of a rational surface singularity $W$, and let $\mc{G}$ be a reflexive sheaf on $W$. 
Then there is an exact sequence
\[
0 \to \Ext^1_Z(\wt{\mc{G}}^Z, \mc{O}_Z) \to \Ext^1_W(\mc{G}, \mc{O}_W) \to \Ext^1_Z(\tor h^*\mc{G}, \mc{O}_Z).
\]
\end{claim}

\begin{proof}[Proof of Claim]
Since $W$ has only rational surface singularities, we have $\mb{R}h_*\mc{O}_Z \cong \mc{O}_W$. By the projection formula in the derived category, we obtain an isomorphism
\[
\Ext^1_W(\mc{G}, \mc{O}_W) \cong \Ext^1_Z(\mb{L}h^*\mc{G}, \mc{O}_Z).
\]
Consider the spectral sequence:
\[
E_2^{p, q} = \Ext^p_Z(\mb{L}_qh^*\mc{G}, \mc{O}_Z) \implies \Ext^{p+q}_Z(\mb{L}h^*\mc{G}, \mc{O}_Z).
\]
For $q > 0$, the $q$-th derived functors $\mb{L}_{q}h^*\mc{G}$ are torsion sheaves supported on the exceptional locus of $h$. Thus, we have $E_2^{0,1} = 0$, yielding an isomorphism
\[
\Ext^1_Z(h^*\mc{G}, \mc{O}_Z) = E_2^{1,0} \cong E^1 \cong \Ext^1_W(\mc{G}, \mc{O}_W).
\]
Applying the functor $\Hom_Z(-, \mc{O}_Z)$ to the exact sequence $0 \to \tor h^*\mc{G} \to h^*\mc{G} \to \wt{\mc{G}}^Z \to 0$, we obtain
\begin{align*}
0&=\Hom_Z(\tor h^*\mc{G}, \mc{O}_Z)\\
\to \Ext^1_Z(\wt{\mc{G}}^Z, \mc{O}_Z) \to \Ext^1_Z(h^*\mc{G}, \mc{O}_Z) &\to \Ext^1_Z(\tor h^*\mc{G}, \mc{O}_Z).
\end{align*}
Via the isomorphism $\Ext^1_Z(h^*\mc{G}, \mc{O}_Z) \cong \Ext^1_W(\mc{G}, \mc{O}_W)$, this gives the desired exact sequence.
\renewcommand{\qedsymbol}{$\triangle$}
\end{proof}

Now, applying the Claim to the (partial) resolutions $g\colon Y\to X$ and $f\colon X\to \Spec R$, we obtain a sequence of injections:
\[
\Ext^1_Y(\wt{M}^Y,\mc{O}_Y)\hookrightarrow
\Ext^1_X(\wt{M}^X,\mc{O}_X)\hookrightarrow
\Ext^1_R(M,R). 
\]
Next, applying the Claim to the minimal resolution $f \circ g\colon Y \to \Spec R$, we have the exact sequence for $M\in \CM (R)$:
\[
0 \to \Ext^1_Y(\wt{M}^Y, \mc{O}_Y) \to \Ext^1_R(M, R) \to \Ext^1_Y(\tor (f\circ g)^*M, \mc{O}_Y).
\]
Since $\mb{R}(f\circ g)_*(f\circ g)^*M\cong M$ and $\mb{R}(f\circ g)_*\wt{M}^Y\cong M$, we have $\mb R (f\circ g)_*(\tor (f\circ g)^*M)=0$. 
Then, Lemma \ref{lem:vanishing_extension} yields $\Ext^1_Y(\tor (f\circ g)^*M, \mc{O}_Y)=0$. 
The exact sequence for $f \circ g$ therefore forces an isomorphism $\Ext^1_Y(\wt{M}^Y,\mc{O}_Y)\cong \Ext^1_R(M,R)$. 
Combined with the injections above, we conclude that all terms are isomorphic:
\[
\Ext^1_Y(\wt{M}^Y,\mc{O}_Y)\cong
\Ext^1_X(\wt{M}^X,\mc{O}_X)\cong
\Ext^1_R(M,R). 
\]
Consequently, $M\in \SCM (R)$ if and only if $\wt{M}^X\in \SCM(X)$. This completes the proof.
\end{proof}

\begin{remark}
In the proof of Lemma \ref{lem:special_partial}, the equivalence $\Ext^1_Y(\wt{M}^Y,\mc{O}_Y)=0 \iff \Ext^1_R(M,R)=0$ also follows directly from  \cite[Theorem 2.2]{MR3006691}. The primary point of the proof above is to establish that this property holds for the minimal partial resolution $X$.
\end{remark}

\begin{example}
Let $R = \C[[x,y]]^G$ be the two-dimensional cyclic quotient singularity of type $\frac{1}{n}(1,1)$, where $G \cong \mathbb{Z}/n\mathbb{Z}$ acts on $\C[[x,y]]$ via $(x,y) \mapsto (\zeta x, \zeta y)$ for a primitive $n$-th root of unity $\zeta$.
Let $f \colon Y \to \operatorname{Spec} R$ be the minimal resolution.
The exceptional set consists of a single smooth rational curve $E \cong \mathbb{P}^1$ with $E^2 = -n$.
\begin{enumerate}
    \item \textbf{Classification of $\ind\CM(R)$ and $\ind\SCM(R)$:} \\
    The indecomposable Cohen--Macaulay (CM) modules over $R$ correspond to the irreducible representations of $G$. There are exactly $n$ such modules up to isomorphism. They are given by the $\zeta^i$-isotypic components of $S$, which are generated as $R$-submodules by the monomials of degree $i$:
    $$M_i = \{ f \in \C[[x,y]] \mid g \cdot f = \zeta^i f \} = R x^i + R x^{i-1}y + \dots + R y^i$$
    for $0 \le i \le n-1$.
    We remark that $\omega_R \cong M_{n-2}$ and $M_i\cong M_{n-i}^\vee$ for $0<i<n$.
    Among these, the indecomposable special Cohen--Macaulay (SCM) modules are in one-to-one correspondence with the exceptional curves, plus the trivial module $R$. Hence, there are exactly two SCM modules:
    \[
    M_0 = R \quad \text{and} \quad M_1 = Rx+Ry.
    \]

    \item \textbf{Reflexive pullbacks:} \\
    For each $0 \le i \le n-1$, the reflexive pullback $\widetilde{M_i}$ is a line bundle on $Y$.
    Let $D_x$ be the strict transform of the divisor $\{x=0\}$ on $Y$, and set $\mc{L} := \mathcal{O}_Y(D_x)$. Then we have $\mc{L} \cdot E = 1$, and $\Pic Y \cong \Span{\mc{L}}$.
    Moreover, we have an isomorphism $\widetilde{M_i} \cong \mc{L}^{\otimes i}$.
    
   \item \textbf{Description of the pushforward $f_* \mathcal{L}^{\otimes i}$:} \\
    Note that $\omega_Y\cong \widetilde{\omega_R} \cong \mc{L}^{\otimes n-2}$. The necessary and sufficient condition for the vanishing of the first extension group is given by:
    \[
    \operatorname{Ext}^1_Y(\mathcal{L}^{\otimes i}, \omega_Y) = 0 \iff  i \le n-1.
    \]
    Then Lemma \ref{lem:pushforward}(i) implies that $f_*\mc{L}^{\otimes i}$ is a reflexive module (in other words, a CM module) for $i \le n-1$. Actually, the pushforward $f_* \mathcal{L}^{\otimes i}$ can be described as follows:
    
    For any integer $i \in \mathbb{Z}$, write $i = qn + r$ with $q \in \mathbb{Z}$ and $0 \le r \le n-1$.
    \begin{itemize}
        \item If $i \ge 0$, then $f_* \mathcal{L}^{\otimes i} \cong \mathfrak{m}_R^q M_r$, where $\mathfrak{m}_R = (x^n, x^{n-1}y, \dots, y^n)R$ is the maximal ideal of $R$. In this case, $f_* \mathcal{L}^{\otimes i}$ is a CM module if and only if $0 \le i \le n-1$ (i.e., $q=0$).
        \item If $i < 0$, let $r'$ be the unique integer $0 \le r' \le n-1$ such that $i \equiv r' \pmod n$. Then $f_* \mathcal{L}^{\otimes i}$ is isomorphic to $M_{r'}$ as an $R$-module. In particular, $f_* \mathcal{L}^{\otimes i}$ is always a CM module for any $i < 0$.
    \end{itemize}
\end{enumerate}
\end{example}


\subsection{Wunram's correspondence for minimal partial resolutions}
In \cite{MR926422}, Wunram established a one-to-one correspondence between the set
$\{E_i\}_{i\in I}$ of exceptional curves of $f\circ g$
and the set of nontrivial indecomposable special Cohen--Macaulay $R$-modules.
He stated this result for two-dimensional quotient singularities over $\mb{C}$.
For completeness, we provide a proof for arbitrary rational surface singularities over an algebraically closed field $k$ of arbitrary characteristic, using results in \cite[\S 3.5]{MR2057015}. 
This result is presumably standard among specialists.
Moreover, we generalize the correspondence to the setting of a minimal partial resolution $f$.

Recall that the correspondence $\mc{L}\in \Pic X\mapsto \mc{L}\cdot E_i\in \mb{Z}$ defines an isomorphism $\Pic X\cong \mb{Z}^{I_f}$. 
Let $\mc{L}_i$ be an invertible sheaf satisfying $\mc{L}_i\cdot E_j=\delta_{ij}$.

Let $\mc{M}_i\in \mk{V}(X)$ be the extension
\[
0\to \mc{O}_X^{\oplus r_i-1}\to \mc{M}_i\to \mc{L}_i\to 0
\]
associated to a minimal set of $r_i-1$ generators for $H^1(X,\mc{L}_i^{-1})$.
Note that $c_1(\mc{M}_i)=\mc{L}_i$.
We see that $\mc{M}_i$ is indecomposable.
Furthermore, any nontrivial indecomposable sheaf in $\mk{V}(X)$ can be constructed in this way (\cite[Proposition 3.5.4]{MR2057015}):
\begin{equation*}
\xymatrix{
\{E_i\}_{i\in I_f} \ar[rr]^{\text{Van den Bergh}\qquad} & & \ind \mk{V}(X)\setminus \{\mc{O}_X\}
}
\end{equation*}

Take a nontrivial indecomposable special Cohen--Macaulay $R$-module $M$.
Since $Y$ is smooth, the reflexive sheaf $\wt{M}^Y$ is locally free, and hence
$\wt{M}^Y\in\mk{V}(Y)$. 
Since
\[
M\cong (f\circ g)_*\wt{M}^Y,
\]
it follows from Lemma~\ref{lem:special_partial} that 
there is a unique indexing $\{M_i\}_{i\in I}$ of the set of nontrivial indecomposable special Cohen--Macaulay $R$-modules such that 
\begin{equation}\label{eqn:Wunram_type}
c_1(\widetilde M_i^Y)\cdot E_j=\delta_{ij}.
\end{equation}
for all $i,j\in I$.

Define the special Cohen--Macaulay module $N_f$ by
\[
N_f:=\bigoplus_{i\in I_f}M_i\oplus R.
\]

Hence, for the minimal resolution $f\circ g$, we obtain the following bijections: 
\[
\begin{tikzcd}[row sep=20pt, column sep=15pt]
  & \{E_i\}_{i\in I} \arrow[dl, leftrightarrow, "\text{Wunram (for }\mb{C})" sloped] \arrow[dr, leftrightarrow, "\text{Van den Bergh}" sloped] & \\
  \ind\SCM(R) \setminus\{R\}= \{M_i\}_{i\in I} & & \ind \mk{V}(Y)\setminus \{\mc{O}_Y\} =\ind \SCM(Y)\setminus\{\mc{O}_Y\}
 \arrow[ll, leftrightarrow, "\text{Lemma \ref{lem:special_partial}}"]
\end{tikzcd}
\]

Let $\{p_s\}$ be the set of singular points of $X$, and 
define
\[
T:=\prod_s \widehat{\mc{O}_{X,p_s}}.
\]
Consider a natural flat morphism $\iota\colon \Spec T\to X$ and take the fiber product $Y_T:=Y\times_X\Spec T$. 
\begin{equation}\label{eqn:resolution_diagram}
\xymatrix{
Y \ar[r]^g           &  X  \ar[r]^f                                     & \Spec R  \\
Y_T \ar[r]^{g_T} \ar[u]_{\iota_Y}          &  \Spec T \ar[u]_{\iota}      &  \\
}
\end{equation}

Every exceptional curve $E_i$ of $g$, i.e., $i\in I_g$, corresponds to an exceptional curve of $g_T$, and we denote it  by $E_i$ for simplicity. 
Then, by the result explained  above, there is a one-to-one correspondence between the set
$\{E_i\}_{i\in I_g}$ of exceptional curves of $g$
and the set
$\{M'_i\}_{i\in I_g}$
of nontrivial indecomposable special Cohen--Macaulay $T$-modules.

\begin{lemma}\label{lem:locallyfree}
Define 
$\wt{M_i}^X|_T:=\iota^*(\wt{M_i}^X)$
for $i\in I$. Then we have
\begin{equation*}
\wt{M_i}^X|_T\cong 
\begin{cases}
 M'_i\oplus F_i & \text{if } i\in I_g, \\
 T^{\oplus \rank M_i} & \text{if } i\in I_f
\end{cases}
\end{equation*}
for some $F_i\in \add T$. Consequently, $\wt{M_i}^X$ is locally free if and only if $i\in I_f$. In other words, $\add \wt{N_f}^X=\mk{V}(X)$.
\end{lemma}

\begin{proof}
We use the Wunram type correspondence \eqref{eqn:Wunram_type} on the minimal resolution. In particular, for each exceptional curve $E_i$, there is a unique non-free indecomposable special CM module $M_i$ whose reflexive pullback has first Chern class satisfying
$
c_1(\widetilde M_i^Y)\cdot E_j=\delta_{ij}.$
We now explain why this correspondence descends to the present partial resolution.

For $i\in I$, it follows from Lemma \ref{lem:pullback} and the flatness of $\iota$ that $\widetilde{M_i}^X|_T$ is reflexive. Moreover, it is a special Cohen--Macaulay $T$-module. Hence 
$$
\widetilde{M_i}^X|_T\cong \bigoplus_{j\in I_g}{M'_j}^{\oplus a_{ij}}\oplus F_i
$$
for some $a_{ij}\ge 0$ and $F_i\in \add T$. We also see 
$(\widetilde{\widetilde{M_i}^X}^Y) |_{Y_T}\cong \widetilde{\widetilde{M_i}^X|_T}^{Y_T}$ by the flatness of $\iota$. 
For $h\in I_g$, we have
\begin{align*}
\delta_{ih}&=c_1(\widetilde{M_i}^Y)\cdot E_h= c_1(\widetilde{\widetilde{M_i}^X}^Y)\cdot E_h
             = c_1(\widetilde{\widetilde{M_i}^X|_T}^{Y_T})\cdot E_h \\
             &= \sum_{j\in I_g}a_{ij}c_1(\widetilde{M'_j}^{Y_T})\cdot E_h=\sum_{j\in I_g}a_{ij}\delta_{jh}=a_{ih}.
\end{align*}
For any $i\in I_g$, $a_{ii}=1$ and $a_{ij}=0$ for $j\neq i$.
For any $i\in I_f$, $a_{ij}=0$ for all $j\in I_g$.
The first claim follows. 

The reflexive sheaf $\widetilde{M_i}^X$ is locally free outside the singular locus of $X$. Hence, it is locally free if and only if $(\widetilde{M_i}^X)_{p_s}$ is a free $\mc{O}_{X,p_s}$-module for all $p_s\in \Sing X$. Since $(-)\otimes_{\mc{O}_{X,p_s}}\widehat{\mc{O}_{X,p_s}}$ is faithfully flat, the second claim follows from the first. 
\end{proof}

We summarize the situation in Proposition \ref{prop:Wunram_partial}, which is a generalization of Wunram's result in \cite{MR926422} for minimal partial resolutions. We remark that this proposition can also be deduced from \cite[Theorem 4.3]{MR3886189}, although it is not explicitly stated in this form there.

\begin{proposition}\label{prop:Wunram_partial}
Let $(R,\mk{m})$ be a $2$-dimensional complete local normal domain with a rational singularity over an algebraically closed field $k$, and let $f\colon X\to \Spec R$ be a minimal partial resolution. Then there is a one-to-one correspondence between
the set $\{E_i\}_{i \in I_f}$ of exceptional divisors of $f$ and
the set of isomorphism classes of nontrivial indecomposable special Cohen--Macaulay $R$-modules $M$ such that $\wt{M}^X$ is locally free on $X$. 
Specifically, this set of modules corresponds to $\{M_i\}_{i \in I_f} \subset \ind\SCM(R)$.
\end{proposition}

\begin{equation}\label{eqn:wunram-type}
\begin{tikzcd}[row sep=15pt, column sep=15pt]
  & \{E_i\}_{i\in I} \arrow[dl, leftrightarrow, "\text{Wunram-type}" sloped] \arrow[dr, leftrightarrow] & \\
  \ind\SCM(R) \setminus\{R\} & & \ind\SCM(X)\setminus\{\mc{O}_X\} \arrow[ll, leftrightarrow, "\text{Lemma \ref{lem:special_partial}}"'{pos=0.25}] \\
  & \{E_i\}_{i\in I_f} \arrow[dl, leftrightarrow, "\text{Proposition \ref{prop:Wunram_partial}}" sloped, pos=0.4] \arrow[dr, leftrightarrow, "\scriptsize\shortstack{Van den Bergh \\ + Lemma \ref{lem:locallyfree}}" sloped, pos=0.4] \arrow[uu, hook] & \\
  \ind(\add N_{f})\setminus\{R\} \arrow[uu, hook] & & \ind(\add \wt{N_f})\setminus \{\mc{O}_X\} \arrow[ll, leftrightarrow] \arrow[uu, hook]
\end{tikzcd}
\end{equation}

In Theorem \ref{thm:DX}, we obtain a categorification of this diagram. 


\section{A Frobenius structure on $\SCM (X)$}\label{sec:Frobenius structure on SCM}
We keep the notation in \S \ref{sec:Wunram-type}. We begin with the following lemma:

\begin{lemma}\label{lem:local-global}
Recall that 
$\iota\colon \Spec T\to X$ is the morphism in the diagram \eqref{eqn:resolution_diagram}, and denote $\iota^*\mc{F}$ by $\mc{F}|_T$.
\begin{enumerate}
\item  For a reflexive sheaf $\mc{F}$ and a sheaf $\mc{G}$ on $X$, the restriction map 
\begin{equation*}
\epsilon\colon \Ext^1_X(\mc{F},\mc{G})\to \Ext^1_T(\mc{F}|_T,\mc{G}|_T)
\end{equation*}
is surjective. Moreover, it is an isomorphism if $\mc{F}\in \SCM (X)$ and $\mc{G}$ is globally generated.
 \item For $\mc{F}\in \CM (X)$ (respectively, $\in \SCM (X)$), we have $\mc{F}|_T\in\CM (T)$  (respectively, $\in \SCM (T)$).
\end{enumerate}
\end{lemma}

 \begin{proof}  
 (i) By the local-global spectral sequence, we have an exact sequence
\begin{align*}
0\to &H^1(X,\mcHom_X(\mc{F},\mc{G}))\to \Ext^1_X(\mc{F},\mc{G})\stackrel{\epsilon'}\to \\
&H^0(X,\mcExt ^1_X(\mc{F},\mc{G}))\to H^2(X,\mcHom _X(\mc{F},\mc{G}))=0.
\end{align*}
Hence, $\epsilon'$ is surjective. 
Assume furthermore that $\mc{F}\in \SCM (X)$ and $\mc{G}$ is globally generated. Then, 
Lemma \ref{lem:surjection} implies that 
$H^1(X,\mcHom_X(\mc{F},\mc{G}))=0$. Therefore,
$\epsilon'$ is an isomorphism.

On the other hand,  note that $\mcExt ^1_X(\mc{F},\mc{G})$ is supported on the singularities of $\Spec T$, and thus we have 
$H^0(X,\mcExt ^1_X(\mc{F},\mc{G}))\cong \Ext^1_T(\mc{F}|_T,\mc{G}|_T)$. 
Under this isomorphism, we have $\epsilon=\epsilon'$. This completes the proof.

 (ii) Since $\iota$ is flat, $\mc{F}|_T$ is reflexive. The remaining conditions follow from (i).
\end{proof}

We define the following full subcategories of $\Coh (X)$, following \cite[\S 3.1]{MR2057015}.
\begin{align*}
\msr{C}&:=\{ \mc{F}\in \Coh (X) \mid \mb{R}f_* \mc{F}=0\},\\
\msr{T}_{-1}&:=\{ \mc{T}\in \Coh (X) \mid \mb{R}^1f_* \mc{T}=0, \Hom_X(\mc{T},\msr{C})=0 \},\\
\msr{F}_{-1}&:=\{ \mc{F}\in \Coh (X) \mid f_* \mc{F}=0\},\\
\msr{T}_{0}&:=\{ \mc{T}\in \Coh (X) \mid \mb{R}^1f_*\mc{T}=0 \},\\ 
\msr{F}_{0}&:=\{ \mc{F}\in \Coh (X) \mid f_* \mc{F}=0, \Hom_X(\msr{C},\mc{F})=0 \}.
\end{align*}

Recall that for $p=-1,0$, the pairs $(\msr{T}_p,\msr{F}_p)$ are torsion theories on $\Coh (X)$ and the heart 
\[
{}^{p}\Per (X/R):=\Span{\msr{F}_p[1],\msr{T}_p}_{\text{ext}}
\]
of the t-structure consists of objects $\mc{G}\in D^b(X)$ whose cohomology sheaves satisfy $\mc{H}^{i}(\mc{G})=0$ for $i\ne 0,-1$, $\mc{H}^{-1}(\mc{G})\in \msr{F}_{p}$, and $\mc{H}^{0}(\mc{G})\in \msr{T}_p$.
Note that a sheaf $\mc{T}$ belongs to $\msr{T}_{-1}$ if and only if it is globally generated (\cite[Lemma 3.1.3]{MR2057015}). 

Moreover, $\wt{N_f}$ (and hence its dual $\wt{N_f}^\vee$) is a tilting bundle on $X$ (see \cite[Proposition 3.2.7]{MR2057015}), and thus it induces triangle equivalences:
\begin{equation}\label{eqn:tilting}
\xymatrix{
\RHom _X(\widetilde{N_f},-)&\colon &D^b(X)\ar[r] &D^b(\mod \End_X(\widetilde{N_f}))\\
&& {}^{-1}\Per (X/R)\ar[r]\ar@{}[u]|{\bigcup} & \mod\End_X(\widetilde{N_f})\ar@{}[u]|{\bigcup} 
}
\end{equation}
\begin{equation}\label{eqn:tilting_dual}
\xymatrix{
\RHom _X(\widetilde{N_f}^\vee,-)&\colon &D^b(X)\ar[r] &D^b(\mod \End_X(\widetilde{N_f}^\vee))\\
&& {}^{0}\Per (X/R)\ar[r]\ar@{}[u]|{\bigcup} & \mod\End_X(\widetilde{N_f}^\vee)\ar@{}[u]|{\bigcup} 
}
\end{equation}
Consequently, we have 
\begin{equation}\label{eqn:projPer} 
\add \widetilde{N_f}=\proj {}^{-1}\Per (X/R).
\end{equation}

\begin{lemma}\label{lem:enough}
Both $\CM (X)$ and $\SCM (X)$ have enough projectives and injectives.   
\end{lemma}

\begin{proof}
Note that $\CM (X)\subset \msr{T}_{-1}\subset {}^{-1}\Per (X/R)$. Hence, \eqref{eqn:projPer} implies that 
$\add \widetilde{N_f}\subset \proj \CM (X)$.

By the tilting equivalence \eqref{eqn:tilting}, given $\mc{F}\in \CM (X)$, there is a short exact sequence
\begin{equation}\label{eqn:SES_proj}
0\to \mc{K}\to \wt{N_f}^{\oplus n}\to \mc{F}\to 0
\end{equation}
in ${}^{-1}\Per (X/R)$. Taking the long exact sequence of cohomology sheaves, we see that $\mc{K}\in \msr{T}_{-1}$, and hence it is globally generated. Moreover, $\mc{K}$ is reflexive since both $\wt{N_f}^{\oplus n}$ and $\mc{F}$ are reflexive.  
Applying $\Hom_X(-,\omega_X)$ to \eqref{eqn:SES_proj} yields the long exact sequence:
\[
\cdots\to\Ext^1_X(\wt{N_f}^{\oplus n},\omega_X)\to \Ext_X^1(\mc{K},\omega_X)\to \Ext_X^2(\mc{F},\omega_X).
\]
Since the first and third terms vanish, $\mc{K}\in \CM (X)$. Hence, \eqref{eqn:SES_proj} is an exact sequence
in the exact category $\CM (X)$, which implies that $\CM (X)$ has enough projectives.

For any $\mc{F}\in \CM (X)$, we have
\begin{align*}
&\RHom_X(\mc{F},\wt{N_f}\otimes \omega_X)\\
\cong& \RHom_X(\mc{F}\otimes \wt{N_f}^\vee, \omega_X) \tag*{\cite[Proposition 5.16]{MR222093}}\\
\cong& \RHom_X(\wt{N_f}^\vee,\mcRHom_X(\mc{F},\omega_X)) \tag*{\cite[Proposition 5.15]{MR222093}}\\
\cong& \RHom_X(\wt{N_f}^\vee,\mcHom_X(\mc{F},\omega_X)).
\end{align*}
Since $\mb{R}^1f_*\mcHom_X(\mc{F},\omega_X)=0$, 
we see that $\mcHom_X(\mc{F},\omega_X)\in \msr{T}_0\subset {}^{0}\Per (X/R)$, and hence 
$\Ext_X^i(\wt{N_f}^\vee,\mcHom_X(\mc{F},\omega_X))=0$ for $i>0$. Therefore, we conclude that $\wt{N_f}\otimes \omega_X\in \inj \CM (X)$.

Furthermore, by the tilting equivalence \eqref{eqn:tilting_dual}, there is a short exact sequence
\begin{equation}\label{eqn:SES_inj}
0\to \mc{K}'\to \wt{N_f}^{\vee \oplus n}\to \mcHom_X(\mc{F},\omega_X)\to 0
\end{equation}
in ${}^{0}\Per (X/R)$. It follows that $\mc{K}'\in \msr{T}_{0}$. Moreover, $\mc{K}'$ is a reflexive sheaf since both $\wt{N_f}^{\vee \oplus n}$ and $ \mcHom_X(\mc{F},\omega_X)$ are reflexive. Applying $\mcHom_X(-,\omega_X)$ to \eqref{eqn:SES_inj}, we obtain a long exact sequence
\[
0\to \mc{F}\to \wt{N_f}^{\oplus n}\otimes \omega_X\to \mcHom_X(\mc{K}',\omega_X)
\to \mcExt_X^1(\mc{F},\omega_X),
\]
where the last term vanishes because $\mc{F}\in \CM (X)$.
Since $\wt{N_f}^{\oplus n}\otimes \omega_X$ is globally generated, its quotient 
$\mcHom_X(\mc{K}',\omega_X)$ is a globally generated reflexive sheaf.
Moreover, we have 
\[
\mb{R}^1f_*\mcHom_X(\mcHom_X(\mc{K}',\omega_X),\omega_X)\cong \mb{R}^1f_*\mc{K}'=0,
\] 
since $\mc{K}'\in \msr{T}_0$.
Thus, $\mcHom_X(\mc{K}',\omega_X)\in \CM (X)$, and therefore the short exact sequence obtained by applying $\mcHom_X(-,\omega_X)$ to \eqref{eqn:SES_inj} lies in $\CM (X)$. This implies that $\CM (X)$ has enough injectives. 

Since $N_{f\circ g}$ is an additive generator of $\SCM (R)$,
$\wt{N_{f\circ g}}$ is an additive generator of $\SCM (X)$ by Lemma \ref{lem:special_partial}. Hence, $\SCM(X)$ is a functorially finite subcategory of $\CM (X)$. Thus, 
\cite[Proposition 4.3]{MR3006691} implies that 
$\SCM (X)$ has enough projectives and injectives.
\end{proof}

\begin{theorem}\label{thm:Frob_SCM}
For the exact categories $\SCM(X)$ and $\CM(X)$, we have the following descriptions of their projective and injective objects:
\begin{enumerate}
    \item $\proj \SCM (X) = \inj \SCM (X) = \add \wt{N_{f\circ g^c}}$.
    \item $\proj \CM (X) = \add \wt{N_f}(=\mk{V}(X))$ and $\inj \CM (X) = \add (\wt{N_f}\otimes \omega_X)(=\mk{V}(X)\otimes \omega_X)$.
\end{enumerate}
Consequently, $\SCM (X)$ has a Frobenius structure.    
\end{theorem}

\begin{proof}
(i) Lemma \ref{lem:locallyfree} implies that for any $M_i'\in \SCM (T)$ (respectively $M_i'\in \add N_{g_T^c}$), there is an $M_i\in \SCM (X)$ (respectively, $M_i\in \add \wt{N_{f\circ g^c}}$) such that $M_i|_T$ is isomorphic to $M_i'$ up to free summands.
On the other hand, 
it follows from \cite[Theorem 4.5]{MR3006691} that $\proj \SCM (T)=\add N_{g_T^c}$.
Moreover, Lemma \ref{lem:local-global} implies that 
$\mc{F}|_T\in \proj \SCM (T)$ for any $\mc{F}\in \proj \SCM (X)$. Therefore, we conclude that $\proj \SCM (X)\subset \add \wt{N_{f\circ g^c}}$. The other direction follows from Lemma \ref{lem:local-global} and the fact that $\wt{N_{f\circ g^c}}|_T\cong N_{g_T^c}$ up to free direct summands. 

By \cite[Lemma 4.4]{MR3006691} and Lemma
\ref{lem:local-global}(i), we obtain an isomorphism 
\[
\Ext_X^1(\mc{F},\mc{G})\cong\Ext_X^1(\mc{G},\mc{F})
\]
for $\mc{F},\mc{G}\in \SCM (X).$
Then the remaining part follows.

(ii) For any $\mc{F}\in \proj \CM (X)$,
consider the exact sequence \eqref{eqn:SES_proj} in $\CM (X)$. 
Then since it splits, we conclude $\mc{F}\in \add \wt{N_f}$, which implies 
$\proj \CM (X)\subset \add \wt{N_f}$. 
The other direction was shown in the proof of Lemma \ref{lem:enough}.

The proof for $\inj \CM (X)$ is similar. 
\end{proof}

We denote the stable category of $\SCM (X)$ by 
\[\underline{\SCM}(X).\]

\begin{remark}\label{rem:IY}
\begin{enumerate}
    \item 
The category $\SCM(R)$ is an exact category whose exact structure is inherited from the abelian category $\mod R$.
Iyama and Wemyss show in \cite{MR3006691} that $\SCM(R)$ has a Frobenius structure with $\proj \SCM(R)=\add N_{(f\circ g)^c}$.
When $f=\id$ (in other words, $X=\Spec R$), the Frobenius structure on $\SCM(X)=\SCM(R)$ given in Theorem \ref{thm:Frob_SCM} coincides with the one given in \cite{MR3006691}.
The stable category associated with this structure is denoted by $\underline{\underline{\SCM}}(R)$ in \cite{MR3006691}, but we use the notation $\underline{\SCM}(R)$.
\item 
It is natural to ask when $\CM (X)$ is a Frobenius category. If this is the case, we have $\proj \CM (X)=\inj \CM (X)$. Then, it follows that $\omega_X$ is an invertible sheaf, and $\omega_X^{-1}\in \add \wt{N_f}$. Thus, both $\omega_X$ and $\omega_X^{-1}$ are globally generated invertible sheaves. Hence, $\omega_X\cong \mc{O}_X$, or equivalently, $\omega_R\cong R$ holds. Consequently, $R$ has at worst a rational double point.
\item 
It is also natural to ask when $\CM (X)=\SCM (X)$. By (ii), this is the case if and only if $R$ has a rational double point. Note also that even when $X$ has at worst rational double points, $\CM (X)$ and $\SCM (X)$ do not necessarily coincide.
\end{enumerate}
\end{remark}

We use the following lemma, which  
is standard for Frobenius categories, to show Proposition \ref{prop:SCMXCMT}. For the convenience of the reader, we include a proof of Lemma \ref{lem:general_result}.
 
\begin{lemma}\label{lem:general_result}
Let $\msr{A}$ be an abelian category and $\msr{C}$ be an extension-closed additive full subcategory of $\msr{A}$. Assume that $\msr{C}$, endowed with the exact structure inherited from $\msr{A}$, is a Frobenius category, and denote its stable category by $\underline{\msr{C}}$. Then we have an isomorphism of abelian groups $\Hom_{\underline{\msr{C}}}(X,Y[1])\cong \Ext^1_\msr{A}(X,Y)$ for $X,Y\in \msr{C}$.  
\end{lemma}

\begin{proof} 
First, note that $\Hom_{\msr{A}}(X, Y)= \Hom_{\msr{C}}(X, Y)$ and $\Ext^1_{\msr{A}}(X, Y) = \Ext^1_{\msr{C}}(X, Y)$ because the exact structure on $\msr{C}$ is induced by that of $\msr{A}$ and $\msr{C}$ is extension-closed. 
By the definition of the shift functor in the stable category, we have an admissible short exact sequence
\[
0 \to Y \to I \to Y[1] \to 0,
\]
where $I$ is an injective-projective object in $\msr{C}$. Applying the functor $\Hom_{\msr{C}}(X, -)$ yields an exact sequence
\[
\Hom_{\msr{C}}(X, I) \xrightarrow{\pi} \Hom_{\msr{C}}(X, Y[1]) \to \Ext^1_{\msr{A}}(X, Y) \to \Ext^1_{\msr{A}}(X, I).
\]
The extension group $\Ext^1_{\msr{A}}(X, I)$ vanishes because $I$ is injective in $\msr{C}$. Thus, we have
\[
\Ext^1_{\msr{A}}(X, Y) \cong \Hom_{\msr{C}}(X, Y[1]) / \Im \pi.
\]
By the definition of the stable category $\underline{\msr{C}}$, the morphism space $\Hom_{\underline{\msr{C}}}(X, Y[1])$ is defined as the quotient of $\Hom_{\msr{C}}(X, Y[1])$ by the ideal of morphisms factoring through injective-projective objects. Since any morphism factoring through an arbitrary injective-projective object also factors through $I$, the image $\Im \pi$ is precisely this ideal. Thus, we obtain the desired isomorphism.
\end{proof}

\begin{proposition}\label{prop:SCMXCMT} For $\iota$ in \eqref{eqn:resolution_diagram},
the pullback functor $\iota ^*$ induces a
 triangle equivalence
\[\iota^*\colon \underline{\SCM}(X)\to  \underline{\SCM}(T)\qquad \mc{F}\mapsto \mc{F}|_T.\]    
\end{proposition}

\begin{proof}
By Lemma \ref{lem:local-global}(ii), the flat morphism $\iota$ induces the triangle functor 
$\iota^*$.  
Note that $\iota^*$ is essentially surjective by Lemma \ref{lem:locallyfree}.
Lemmas \ref{lem:local-global}(i) and \ref{lem:general_result} ensure the full faithfulness of $\iota^*$. 
\end{proof}


\section{An equivalence between stable categories of special Cohen--Macaulay sheaves}
Throughout this section, assume that $X$ has rational double points. 
Consequently, $g^c=\id_X$. 
This assumption is essential to equip $\SCM_{N_f} (R)$ with the Frobenius structure introduced in \cite{MR3320570}. Under this condition, Theorem \ref{thm:Frob_SCM} implies that $N_{f\circ g^c} = N_f$, reflecting the fact that Theorem \ref{thm:DX} serves precisely as a categorification of the three-way correspondence illustrated in diagram \eqref{eqn:wunram-type}.

\subsection{Frobenius structures introduced by Iyama--Kalck--Wemyss--Yang}
Following \cite{MR3320570}, we recall an exact structure on $\SCM (R)$. 
For objects $L_1, L_2, L_3 \in \SCM(R)$, consider a short exact sequence
\begin{equation}\label{eqn:L123}
0\to L_1\stackrel{\alpha}\to L_2 \stackrel{\beta}\to L_3 \to 0
\end{equation}
such that $\Hom_R(N_f, \beta)$ is surjective. Such sequences constitute the (admissible) exact sequences in the exact category $\SCM(R)$.
We denote this exact category by \[\SCM_{N_f}(R).\]
It is shown in \cite{MR3320570} that $\SCM_{N_f}(R)$ has a Frobenius structure with $\proj \SCM_{N_f} (R)=\add N_f$. We denote its stable category by
\[\underline{\SCM}_{N_f}(R).\]

\begin{remark}
In our context, we consider three exact categories: $\SCM_{N_f}(R)$, $\SCM(X)$, and $\CM (R)$. While the exact structures of $\SCM(X)$ and $\CM (R)$ are inherited from $\Coh (X)$ and $\mod R$, respectively, that of $\SCM_{N_f}(R)$ is not inherited from $\mod R$ if $f\ne f^c$.

In the case where $f=f^c$, it follows from \cite[Theorem 4.5]{MR3006691} that $\Ext^i_R(N_{f^c},L)=0$ for all $i>0$ and $L\in\SCM (R)$. Thus, for a given short exact sequence 
\[0\to L_1\to L_2\stackrel{\beta}\to L_3\to 0\]
for $L_i\in \SCM (R)$,
$\Hom_R(N_{f^c},\beta)$ is always surjective.
It turns out that the exact structure of $\SCM_{N_{f^c}}(R)$ is inherited from $\mod R$. Consequently, by definition, $\underline{\SCM}_{N_{f^c}} (R)=\underline{\SCM} (R)$.
\end{remark}

We obtain the following categorification of the diagram \eqref{eqn:wunram-type}.


\begin{theorem}\label{thm:DX}
The assignment $M \mapsto \wt{M}$ yields an equivalence of exact categories
\[
\SCM_{N_f} (R)\xrightarrow{\sim} \SCM (X)
\]
with quasi-inverse $f_*$, which restricts to an equivalence $\add N_f\xrightarrow{\sim} \mk{V}(X)$. 
Consequently, this equivalence induces a triangle equivalence 
    \[
    \underline{\SCM}_{N_f}(R)\xrightarrow{\sim} \underline{\SCM} (X).
    \]
\end{theorem}

\begin{proof}
First of all, note that 
\begin{equation}\label{eqn:N=wtN}
\Hom_R(L,M)\cong\Hom_R(L,f_*\widetilde{M})
\cong\Hom_X(f^*L,\widetilde{M})
\cong\Hom_X(\widetilde{L},\widetilde{M}),
\end{equation}
which implies that the additive functor $\widetilde{(-)}$ is fully faithful.
Lemma \ref{lem:special_partial} yields that 
$\widetilde{(-)}$ and $f_*$ give an equivalence of additive categories.
The statement on $\add N_f\xrightarrow{\sim} \mk{V}(X)$ follows from Lemma \ref{lem:locallyfree}.

Given an exact sequence \eqref{eqn:L123}, 
we obtain a short exact sequence
\[
0\to \Hom_X (\wt{N_f},\wt{L_1})\to \Hom_X (\wt{N_f},\wt{L_2}) \to \Hom_X (\wt{N_f},\wt{L_3}) \to 0
\]
by \eqref{eqn:N=wtN}.
Since $\wt{L_i}\in {}^{-1}\Per (X/R)$, we obtain a short exact sequence
\[ 
0\to \wt{L_1}\to \wt{L_2} \to \wt{L_3} \to 0.
\]

On the other hand, suppose that we have a short exact sequence
\[
0\to \mc{F}_1 \to \mc{F}_2\to \mc{F}_3 \to 0
\]
in $\Coh (X)$ with $\mc{F}_i\in \SCM (X)$.
Then, since $\mb{R}^1f_*\mc{F}_1=0$, we obtain a short exact sequence
\[
0\to f_*\mc{F}_1 \to f_*\mc{F}_2\to f_*\mc{F}_3 \to 0,
\]
which is easily seen to be an admissible exact sequence in $\SCM_{N_f}(R)$. This completes the proof of the equivalence of exact categories. The remaining assertions follow from this and the last part of Lemma \ref{lem:locallyfree}.
\end{proof}


\subsection{Equivalence between the triangulated categories $\underline{\SCM}_{N_f}(R)$ and $\underline{\CM}(T)$}
The existence of an equivalence between the triangulated categories $\underline{\SCM}_{N_f}(R)$ and $\underline{\CM}(T)$ was established in characteristic $0$ in \cite[Theorem 1.4]{MR3320570}.
Using Theorem \ref{thm:DX}, we provide a more geometric proof (in arbitrary characteristic) than that of \cite{MR3320570}. Furthermore, in \S \ref{subsec:comparison}, we remark that the equivalence constructed here is naturally isomorphic to theirs.

We define the composed triangle functor
\[\Phi := \iota^* \circ \wt{(-)} \colon \underline{\SCM}_{N_f}(R) \to \underline{\CM}(T), \quad M \mapsto \wt{M}|_T.
\]

Combining Proposition \ref{prop:SCMXCMT} with Theorem \ref{thm:DX}, we obtain the following.

\begin{corollary}[{\cite[Theorem 1.4]{MR3320570}}]\label{cor:IKWY}
The functor $\Phi\colon \underline{\SCM}_{N_f}(R) \to \underline{\CM}(T)$ is a triangle equivalence.    
\end{corollary}

\subsection{Remarks on the proof and comparison with \cite{MR3320570}}\label{subsec:comparison}

In \cite{MR3320570}, the authors establish an equivalence between the stable categories $\underline{\SCM}_{N_f}(R)$ and $\underline{\CM}(T)$ by composing a sequence of four equivalences. As depicted in the top row of the diagram below, their approach is largely algebraic and passes through singularity categories. Specifically, \cite[Theorem 3.8]{MR3320570} utilizes the Frobenius structure on $\SCM_{N_f}(R)$ to prove that $\End_R(N_f)$ is an Iwanaga--Gorenstein ring. The general theory of such rings ensures that the category of Gorenstein projective modules $\GP(\End_R(N_f))$ is a Frobenius category, yielding a triangle equivalence $\Hom_R(N_f,-) \colon \underline{\SCM}_{N_f}(R) \xrightarrow{\sim} \underline{\GP}(\End_R(N_f))$. From there, the sequence proceeds to $\underline{\CM}(T)$ via the singularity categories $D_{\sg}(\End_R(N_f))$ and $D_{\sg}(X)$ using Buchweitz's and Orlov's theorems \cite[Theorem 1.4]{MR3320570}.

\begin{equation*}
\xymatrix@C=1.2em{
\underline{\SCM}_{N_f}(R) \ar[r]_(0.4){\sim} \ar@{=}[d] & 
\underline{\GP}(\End_R(N_f)) \ar[rr]^{\raisebox{0.8ex}{\text{\scriptsize Buchweitz}}}_{\sim} 
& & 
D_{\sg}(\End_R(N_f)) \ar[r]_(0.6){\sim} & 
D_{\sg}(X) \ar[rr]^{\raisebox{0.8ex}{\text{\scriptsize Buchweitz+Orlov}}}_{\sim} & & 
\underline{\CM}(T) \ar@{=}[d] \\
\underline{\SCM}_{N_f}(R) \ar[r]^{\sim}_{\wt{(-)}} & 
\underline{\SCM}(X) \ar[rrrrr]^{\sim}_{\iota^*} & & & & & 
\underline{\CM}(T)
}
\end{equation*}

In contrast, our proof provides a more direct, geometric route, bypassing the use of singularity categories entirely, as indicated by the bottom row of the diagram.

We remark that this diagram commutes up to natural isomorphism, which can be verified by tracing the precise definitions of the equivalences in the top row. This commutativity and the relationship between the two rows yield a few notable consequences. For instance, the isomorphism $\End_R(N_f) \cong \End_X(\wt{N_f})$ from \eqref{eqn:N=wtN} implies that the tilting equivalence $\RHom_X(\widetilde{N_f},-)$ in \eqref{eqn:tilting} restricts to an equivalence $\SCM(X) \simeq \GP(\End_R(N_f))$. Furthermore, the diagram immediately yields the following Buchweitz-type equivalence.

\begin{theorem}\label{thm:buchweitz}
There is a triangle equivalence between $\underline{\SCM}(X)$ and $D_{\sg}(X)$.
\end{theorem}

\begin{remark}
In Theorem \ref{thm:buchweitz}, both of $\underline{\SCM}(X)$ and $D_{\sg}(X)$ can be defined without assuming $X$ has at most rational double points. Suppose that $X$ does not necessarily have rational doble points. Then we obtain $\underline{\SCM}(X)\simeq D_{\sg}(X^c)$ by Proposition  \ref{prop:SCMXCMT} and Theorem \ref{thm:buchweitz}.       
\end{remark}

We also obtain the following consequence.
\begin{corollary}
Let $R$ be a rational double point, and $f\colon X\to \Spec R$ be a minimal partial resolution. Then the triangulated category $D_{\sg} (X)$ is obtained by the additive quotient of $D_{\sg} (R)$ by $\add N_f$.    
\end{corollary}

\begin{proof}
For a rational double point $R$, we have $\underline{\SCM}(R)=\underline{\CM}(R)$. On the other hand, we have
\[\underline{\SCM}_{N_f}(R)=\SCM(R)/\add N_f \simeq \underline{\SCM}(R)/\add N_f \simeq \underline{\CM}(R)/\add N_f,\]
where we regard $N_f$ as an object of $\underline{\SCM}(R)=\underline{\CM}(R)$ in the last two terms. Here, all quotients denote the additive quotients by the ideal of morphisms factoring through $\add N_f$.
Applying Theorem \ref{thm:DX} and Theorem \ref{thm:buchweitz}, we obtain triangle equivalences 
\[\underline{\CM}(R)/\add N_f \simeq \underline{\SCM}_{N_f}(R) \simeq \underline{\SCM}(X) \simeq D_{\sg}(X).\]
Since the equivalence $\underline{\CM}(R)\simeq D_{\sg}(R)$ is induced by the natural functor $\CM(R)\hookrightarrow D^b(R)$, we conclude that $D_{\sg}(X) \simeq D_{\sg}(R)/\add N_f$.
\end{proof}

It is an interesting question whether a similar picture holds for other kinds of singularities.

\bibliographystyle{plain}
\bibliography{non-standard}

\noindent
Hokuto Uehara

Department of Mathematical Sciences,
Graduate School of Science,
Tokyo Metropolitan University,
1-1 Minamiohsawa,
Hachioji,
Tokyo,
192-0397,
Japan 

{\em e-mail address}\ : \  hokuto@tmu.ac.jp
\ \vspace{0mm} \\
\end{document}